\documentclass[12pt]{article}

\usepackage{amsmath,amsthm,amssymb}
\usepackage{fullpage,setspace}

\newtheorem{theorem}{Theorem}[section]
\newtheorem{lemma}[theorem]{Lemma}
\newtheorem{corollary}[theorem]{Corollary}
\newtheorem{proposition}[theorem]{Proposition}

\theoremstyle{definition}
\newtheorem{definition}{Definition}[section]

\def\R{\mathbb{R}}
\def\C{\mathbb{C}}
\def\eps{\varepsilon}
\def\Re{\operatorname{Re}}
\def\Im{\operatorname{Im}}
\def\deg{\operatorname{deg}}
\def\diam{\operatorname{diam}}

\def\endproof{{\mbox{}\nolinebreak\hfill\rule{2mm}{2mm}\par\medbreak}}

\begin{document}

\title{An observation on Tur\'an-Nazarov inequality}

\author{O. Friedland \and Y. Yomdin}

\newcommand\address{\noindent\leavevmode\noindent \\
O. Friedland, \\
Institut de Math\'ematiques de Jussieu, \\
Universit\'e Pierre et Marie Curie (Paris 6) \\
4 Place Jussieu, \\
75005 Paris, France \\
\texttt{e-mail: \small friedland@math.jussieu.fr} \\
\medskip \\
\noindent
Y. Yomdin, \\
Department of Mathematics, \\
The Weizmann Institute of Science, \\
Rehovot 76100, Israel \\
\texttt{e-mail: \small yosef.yomdin@weizmann.ac.il}
}

\date{}

\footnotetext[1]{Primary Classification 26D05 Secondary Classification 30E05, 42A05}
\footnotetext[3]{The research of the second author was supported by ISF grant No. 264/09 and by the Minerva Foundation.}

\maketitle

\begin{abstract}
The main observation of this note is that the Lebesgue measure $\mu$ in the Tur\'an-Nazarov inequality for exponential polynomials
can be replaced with a certain geometric invariant $\omega \ge \mu$, which can be effectively estimated in terms of the metric
entropy of a set, and may be nonzero for discrete and even finite sets. While the frequencies (the imaginary parts of the exponents)
do not enter the original Tur\'an-Nazarov inequality, they necessarily enter the definition of $\omega$.
\end{abstract}

\bigskip

\section{Introduction}\label{intro}

The classical Tur\'an inequality bounds the maximum of the absolute value of an exponential polynomial $p(t)$ on an interval $B$
through the maximum of its absolute value on any subset $\Omega$ of positive measure. Tur\'an \cite{T} assumed $\Omega$ to be a
subinterval of $B$, and Nazarov \cite{N} generalized it to any subset $\Omega$ of positive measure. More precisely, we have:

\begin{theorem}[\cite{N}]\label{turan}
Let $p(t) = \sum_{k=0}^m c_k e^{\lambda_k t}$ be an exponential polynomial, where $c_k ,\lambda_k \in \C$. Let $B \subset \R$ be
an interval, and let $\Omega \subset B$ be a measurable set. Then
$$
\sup_{B} |p| \le e^{\mu_1(B) \cdot \max|\Re \lambda_k|} \cdot \left( \frac{c \mu_1(B)}{\mu_1(\Omega) }\right)^{m} \cdot \sup_{\Omega} |p|
$$
where $\mu_1$ is the Lebesgue measure on $\R$ and $c>0$ is an absolute constant.
\end{theorem}

In this note, we generalize and strengthen the Tur\'an-Nazarov inequality (and its multi-dimensional analogue stated below) by replacing
the Lebesgue measure of $\Omega$ with a simple geometric invariant $\omega_D(\Omega)$. We call it the metric span of $\Omega \subset {\mathbb R}^n$
with respect to a ``diagram'' $D$ of $p$ comprising the degree of $p$ and its maximal frequency $\lambda$. The metric span always bounds the Lebesgue
measure from above, and it is strictly positive for sufficiently dense discrete (in particular, finite) sets $\Omega$. It can be effectively
estimated in terms of the metric entropy of $\Omega$. See \cite{Y} and Section \ref{examples} below for some basic properties of $\omega_D(\Omega)$.

\smallskip

Our approach is as follows: Put $\rho=\sup_{\Omega} |p|$. Then $\Omega \subset V_\rho$, where $V_\rho = V_\rho(p)=\{t\in B, \ |p(t)|\leq \rho \}$
is the $\rho$-sublevel set of the exponential polynomial $p$. Next we use a theorem of Khovanskii in \cite{K} to give an upper bound on the
number of solutions of $|p(t)|= \rho$ in an interval $B$ in terms of the length of the interval, the degree of $p$ and the maximal frequency
of $p$. This also bounds from above the number of intervals in $V_\rho$. Next, for $V_\rho$, consisting of a finite number of closed intervals,
it is easy to compare the Lebesgue measure $\mu_1(V_\rho)$ and the metric entropy of $\Omega \subset V_\rho$. We conclude that
$\mu_1(V_\rho)\geq \omega_D(\Omega).$ Finally, we apply the original Tur\'an-Nazarov inequality of Theorem \ref{turan} to the subslevel set $V_\rho$.

\smallskip

With appropriate modifications this approach works also in higher dimensions. Originally it was applied in \cite{Y} in order to produce a Remez-type inequality for algebraic polynomials on discrete sets. The corresponding invariant $\omega_{n,d}(\Omega)$ depends only on the dimension and the degree,
and uses Vitushkin's bound (see \cite{V}, and \cite{FY} for further developments in this direction) for the metric entropy of semialgebraic sets instead of the Khovanskii's bound. It replaces the Lebesgue measure of $\Omega$ in the classical Remez inequality for algebraic polynomials (\cite{Rem,Bru.Gan}).

\medskip

Now we give an accurate statement of our main results in one-dimensional case. For a given exponential polynomial 
$p(t) = \sum_{k=0}^m c_k e^{\lambda_k t}$ with $c_k ,\lambda_k \in \C$, and for a given interval $B \subset \R$, we define the diagram 
$D=D(p,B)=(m,\lambda,l)$. It comprises the degree $m$ of $p$, the maximal frequency $\lambda=\displaystyle \max_{k=0,\dots,m} |\Im \lambda_k|$, and 
the length $l=\mu_1(B)$.

Define the constant $M_D$ (which we call a ``frequency bound" for $p$) as $M_D = \lfloor {d \over 2} \rfloor + 1$, where $d=d(m,\lambda,l)$ is the
maximal number of solutions of $|p|=\rho, \ \rho \in \R,$ on an interval of length $l$, for a complex exponential polynomial $p$ of degree $m$ and of
maximal frequency $\lambda$.

For any bounded subset $\Omega \subset {\mathbb R}$ and for $\epsilon > 0$ let $M(\epsilon,\Omega)$ be the minimal number of $\epsilon$-intervals
covering $\Omega$. Now the metric span $\omega_D$ is defined as follows:

\begin{definition}\label{span}
The metric span $\omega_D(\Omega)$ of $\Omega \subset {\mathbb R}$
is given by
$$
\omega_{D}(\Omega) = \sup_{\eps > 0} \eps [M(\eps , \Omega) - M_D]
$$
\end{definition}

Now we can state our main result in one-dimensional case:

\begin{theorem}\label{d.turan}
Let $p(t) = \sum_{k=0}^m c_k e^{\lambda_k t}$ be an exponential polynomial, where $c_k ,\lambda_k \in \C$. Let $B \subset \R$ be an interval, and
let $\Omega \subset B$ be any set. Then
$$
\sup_{B} |p| \le e^{\mu_1(B) \cdot \max|\Re \lambda_k|} \cdot \left( \frac{c \mu_1(B)}{\omega_D(\Omega) }\right)^{m} \cdot \sup_{\Omega} |p|
$$
where $c>0$ is an absolute constant.
\end{theorem}
Using Khovanskii's bound in \cite{K} we can give more explicit (although somewhat cumbersome) expression for $d$, and hence for $M_D$ and $\omega_D$.
Let us put $\tilde d= \tilde d(m,\lambda,l)=C(m) l \lambda$. Here $C(m)$ is defined as $C(m)=n(2n+1)^{2n}2^{2n^2}$, for $n={{(m+1)(m+2)}\over 2}+1$.
Next we define $\tilde M_D = \lfloor {{\tilde d} \over 2} \rfloor + 1,$ and $\tilde \omega_{D}(\Omega) = \sup_{\eps > 0} \eps [M(\eps , \Omega) - \tilde M_D].$\
As we shall see below, always $d\leq \tilde d$, and hence $\tilde \omega_{D}(\Omega) \leq \omega_{D}(\Omega)$.

\begin{corollary}\label{Khov}
Under conditions of Theorem \ref{d.turan}
$$
\sup_{B} |p| \le e^{\mu_1(B) \cdot \max|\Re \lambda_k|} \cdot \left( \frac{c \mu_1(B)}{\tilde \omega_D(\Omega) }\right)^{m} \cdot \sup_{\Omega} |p|.
$$
\end{corollary}

\noindent{\bf Remark 1.} The same type of reasoning applies to any class of functions for which a Remez-type inequality and a uniform bound on the
number of zeroes hold.

\smallskip

\noindent{\bf Remark 2.} For any \textit{measurable} $\Omega$ we always have $\omega_{D}(\Omega) \ge \mu_{1}(\Omega)$, with equality if $\Omega$ is a
sublevel set of $p$ (see Section \ref{examples1} below). Thus, Theorem \ref{d.turan} provides a true generalization and strengthening of the
Tur\'an-Nazarov inequality given in Theorem \ref{turan}.

\smallskip

\noindent{\bf Remark 3.} We insist in Definition \ref{span} above that $\omega_D$ depends only on the {\it imaginary parts} of the exponents $\lambda_k$,
i.e. on the frequencies (and consequently we get a rather complicated bound in Corollary \ref{Khov}. Compare Theorems \ref{d.turan1}, \ref{d.turan2} below).

But this separation allows us to preserve and further develop a remarkable feature of the original Tur\'an-Nazarov inequality: The bound does not depend
on the frequencies, i.e. on the imaginary parts of $\lambda_k$ in $p$. When we allow into consideration \textit{discrete} (in particular, \textit{finite})
sets $\Omega$, this feature certainly cannot be completely preserved: Already for a trigonometric polynomial $p(t)= \sin (\lambda t)$, the set $\Omega$ of
its zeroes (on which the Tur\'an-Nazarov inequality certainly fails) consists of all the points $x_j= {{j\pi}\over \lambda}, \ j \in {\mathbb N}$, and the
number of such points in any interval $B$ is of order ${{\mu(B)\lambda}\over \pi}$.

However, Theorem \ref{d.turan} separates the roles of the real and imaginary parts of the exponents: The first enter the main bound, as in the original
Tur\'an-Nazarov inequality, while the second enter the definition of the span $\omega_D(\Omega)$. As the density of $\Omega$ growth, the influence of the
frequencies decreases: See Section \ref{examples} below.

\smallskip

\noindent{\bf Remark 4.} Recently promising applications of Theorem \ref{d.turan} have been found in Signal Processing, specifically, in non-uniform exponential sampling (see (\cite{S,Bat.Yom,Bat.Sar.Yom}) and references therein).

\medskip

There is a version of Tur\'an-Nazarov inequality for quasipolynomials in one or several variables due to A. Brudnyi \cite[Theorem 1.7]{B}. While less
accurate than the original one (in particular, the role of real and complex parts of the exponents is not separated) this result gives an important
information for a wider class of quasipolynomials. In Section \ref{multi.dim} we provide a strengthening of Brudnyi's result in the same lines as above:
We replace the Lebesgue measure with an appropriate ``metric span" which always bounds the Lebesgue measure from above and is strictly positive for
sufficiently dense discrete (in particular, finite) sets.

\smallskip

The authors would like to thank the referee for the remarks and suggestions, significantly improving the presentation. In particular, ``an equality'' 
part of Proposition \ref{Arithm} in Section \ref{examples} was suggested by the referee.

\section{Proofs and examples in dimension one}\label{one.dim}

In this section we prove Theorem \ref{d.turan} and provide some of its consequences.

\medskip

\noindent {\bf Proof of Theorem \ref{d.turan}.} Let $p(t) = \sum_{k=0}^m c_k e^{\lambda_k t}$ be an exponential polynomial, $c_k ,\lambda_k \in \C$.
Let $B \subset \R$ be an interval. We consider the sublevel set $V_\rho = \{ t \in B : \vert p(t) \vert \le \rho \}$ of $p(t)$. By definition,
$d=d(m,\lambda,\mu_1(B))$ is the maximal number of solutions of $|p|=\rho, \ \rho \in \R,$ on the interval $B$. Hence the boundary of $V_\rho$ consists
of at most $d$ points (including the endpoints). Therefore, the set $V_\rho$ consists of at most $M_D=\lfloor {d \over 2} \rfloor + 1$ subintervals
$\Delta_i$ (i.e. connected components of $V_\rho$), with $M_D$ defined as in Theorem \ref{d.turan}. Let us cover each of these subinterval $\Delta_i$
by the adjacent $\eps$-intervals $Q_{\eps}$ starting with the left endpoint. Since all the adjacent $\eps$-intervals, except possibly one, are inside
$\Delta_i$, their number doesn't exceed $|\Delta_i| / \eps + 1$. Thus, we have
$$
M(\eps,V_\rho) \le (\lfloor {d \over 2} \rfloor + 1) + \mu_1 (V_\rho) / \eps=M_D+\mu_1 (V_\rho) / \eps.
$$
Now let a set $\Omega \subset B$ be given.

\begin{lemma}\label{meas.span}
If $\Omega \subset V_\rho$ for a certain $\rho \geq 0$ then $\mu_1(V_\rho) \geq \omega_D(\Omega)$.
\end{lemma}

\begin{proof}
If $\Omega \subset V_\rho$ then for each $\eps > 0$ we have $M(\eps,\Omega)\leq M(\eps,V_\rho) \le M_D + \mu_1 (V_\rho) / \eps$, or
$\mu_1 (V_\rho)\ge \eps[M(\eps,\Omega)-M_D]$. Taking supremum with respect to $\eps > 0,$ via Definition \ref{span} we conclude that
$\mu_1(V_\rho) \geq \omega_D(\Omega)$.
\end{proof}

Let us now put $\hat \rho = \displaystyle{\sup_{\Omega}} |p|$. Then we have $\Omega \subset V_{\hat \rho}$. Applying Lemma \ref{meas.span}
we get $\mu_1(V_{\hat \rho}) \geq \omega_D(\Omega)$. Finally, we apply the original Tur\'an-Nazarov inequality (Theorem \ref{turan}) to the subset
$V_{\hat \rho} \subset B$ on which $|p|$ by definition does not exceed $\hat \rho$. This completes the proof of Theorem \ref{d.turan}. \endproof

\medskip

\noindent {\bf Proof of Corollary \ref{Khov}.} Let, as above, $p(t) = \sum_{k=0}^m c_k e^{\lambda_k t}$ be an exponential polynomial, where
$c_k ,\lambda_k \in \C$. Let us write $c_k=\gamma_k e^{i\phi_k}, \ \lambda_k = a_k+ib_k, \ k=0,1,\dots,m$.

\begin{lemma}\label{abs.val.p}
$$
\vert p(t) \vert^2 =2\sum_{0\le k\le l \le m}\gamma_k \gamma_l e^{(a_k+a_l)t} \cos(\phi_k - \phi_l + (b_k-b_l)t)
$$
is an exponential-trigonometric polynomial of degree ${{(m+1)(m+2)}\over 2}$ with real coefficients.
\end{lemma}

\begin{proof}
We have
$$
p(t) = \sum_{k=0}^m \gamma_k e^{i\phi_k} e^{(a_k+ib_k)t}= \sum_{k=0}^m \gamma_k e^{a_kt+i(\phi_k+b_kt)}, \ \bar p(t)= \sum_{k=0}^m \gamma_k e^{a_kt-i(\phi_k+b_kt)}
$$

Therefore
$$
\vert p(t) \vert^2 = p(t)\bar p(t) = \sum_{k,l=0}^m \gamma_k \gamma_l e^{(a_k+a_l)t+i(\phi_k-\phi_l +(b_k-b_l)t)}
$$

Adding the expressions in this sum for the indices $(k,l)$ and $(l,k)$ we get
$$
\vert p(t) \vert^2 =2\sum_{k\le l}\gamma_k \gamma_l e^{(a_k+a_l)t} \cos(\phi_k - \phi_l + (b_k-b_l)t)
$$

This completes the proof.
\end{proof}

\medskip

The following lemma provides us with a bound on the number of real solutions of the equation $\vert p(t) \vert=\rho$. It is a direct consequence of
Khovanskii's bound Theorem \ref{thm:khovanskii} and Lemma \ref{khovanskii} in Section \ref{covering.number} below.

\begin{lemma}\label{Khov.bound.1}
For $p(t)$ as above and for each positive $\eta>0$, the number of non-degenerate solutions of the equation $\vert p(t) \vert=\eta$ in the interval
$B\subset {\mathbb R}$ does not exceed
$$
\tilde d=C(m) \mu_1(B)\lambda
$$
where $\lambda= \max |\Im \lambda_k|$, and $C(m)=n(2n+1)^{2n}2^{2n^2}$, for $n={{(m+1)(m+2)}\over 2}+1$.
\end{lemma}
So we have $d\leq \tilde d, \ M_d \leq M_{\tilde d}, \ \omega_D(\Omega) \geq \tilde \omega_D(\Omega).$ This completes the proof of Corollary \ref{Khov}.
\endproof

\medskip

We expect that the expression for $C(m)$ in Lemma \ref{Khov.bound.1} provided by the general result of Khovanskii can be strongly improved in our 
specific case. Let us recall the following result of Nazarov \cite[Lemma 4.2]{N}, which gives a much more realistic bound on the local distribution 
of zeroes of an exponential polynomial if the real parts of its exponents are relatively small:

\begin{lemma}\label{naz}
Let $p(t) = \sum_{k=0}^m c_k e^{\lambda_k t}$ be an exponential polynomial, $c_k ,\lambda_k \in \C$. Then the number of zeroes of $p(z)$ inside each 
disk of radius $r>0$ does not exceed $4m+7\hat \lambda r$, where $\hat
\lambda = \max|\lambda_k|$.
\end{lemma}

The reason we use the Khovanskii bound in Theorem \ref{d.turan} is that it involves only the imaginary parts of the exponents $\lambda_k$. In contrast, 
the bound of Lemma \ref{naz} is in terms of $\hat \lambda = \max|\lambda_k|$ (as opposed to $\max|\Im \lambda_k|$). So for the real parts of the 
exponents of $p$ large, the Khovanskii bound may be better. 

In order to apply Lemma \ref{naz} we notice that
$$
\vert p(t) \vert^2 = p(t)\bar p(t) = \sum_{k,l=0}^m c_k \bar c_l e^{(\lambda_k + \bar \lambda_l) t}
$$
is an exponential polynomial of degree at most $m^2$ with the maximal absolute value of the exponents not exceeding $2\hat \lambda$. Adding a constant 
adds at most one to the degree. We conclude that the number of real solutions of $\vert p(t) \vert =\eta$ inside the interval $B$ does not exceed 
$d_1=4m^2+14\hat \lambda \mu_1(B)$. Now we define $\omega'_D$ putting $M'_D=\lfloor {d_1 \over 2} \rfloor + 1$ in Definition \ref{span}.
Repeating word by word the proof of Theorem \ref{d.turan} above we obtain:

\begin{theorem}\label{d.turan1}
For $p(t)$ as above
$$
\sup_{B} |p| \le e^{\mu_1(B) \cdot \max|\Re \lambda_k|} \cdot \left( \frac{c \mu_1(B)}{\omega'_D(\Omega) }\right)^{m} \cdot \sup_{\Omega} |p| .
$$
\end{theorem}

For the case of a real exponential polynomial $p(t) = \sum_{k=0}^m c_k e^{\lambda_k t}$, $c_k ,\lambda_k \in \R$, we get an especially simple and sharp result. Notice that the number of zeroes of a real exponential
polynomial is always bounded by its degree $m$ (indeed, the ``monomials" $e^{\lambda_k t}$ form a Chebyshev system on each real interval). Applying this fact in the same way as above we get

\begin{theorem}\label{d.turan2}
For $p(t)$ a real exponential polynomial of degree $m$
$$
\sup_{B} |p| \le e^{\mu_1(B) \cdot \max|\lambda_k|} \cdot \left( \frac{c \mu_1(B)}{\omega''_D(\Omega) }\right)^{m} \cdot \sup_{\Omega} |p|
$$
where $\omega''_D(\Omega)= \sup_{\eps>0} \eps [M(\eps,\Omega)-m]$.
\end{theorem}

Notice that in this case the metric span $\omega''_{D}(\Omega)$ depends only on the degree $m$ of $p$ and the result is sharp: For any $\Omega$
consisting of at least $m+1$ points there is an inequality of the required form, while for each $m$ points there is a real exponential polynomial
$p(t)$ of degree $m$ vanishing at exactly these points.

\subsection{Some examples}\label{examples}

In this section we give just a couple of examples illustrating the properties of the span $\omega_D$, as well as the scope and possible
applications of Theorem \ref{d.turan}.

\subsubsection{$\omega_D(\Omega)$ versus $\mu_1(\Omega)$}\label{examples1}

Let us recall that for a given interval $B$ and for an exponential polynomial $p(t) = \sum_{k=0}^m c_k e^{\lambda_k t}$, $c_k ,\lambda_k \in \C$,
its diagram $D=D(p,B)=(m,\lambda,l)$ comprises the degree $m$ of $p$, the maximal frequency $\lambda=\displaystyle \max_{k=0,\dots,m} |\Im \lambda_k|$,
and the length $l=\mu_1(B)$. Next, $d=d(m,\lambda,l)$ is the maximal number of solutions of $|p|=\rho, \ \rho \in \R,$ on an interval of length $l$,
$M_D = \lfloor {d \over 2} \rfloor + 1$, and $\omega_{D}(\Omega) = \sup_{\eps > 0} \eps [M(\eps , \Omega) - M_D]$.

\begin{proposition} \label{Arithm}
For any measurable $\Omega$ we have $\omega_{D}(\Omega) \ge \mu_{1}(\Omega)$, with equality if $\Omega=V_\rho$ is a sublevel set of $p$.
\end{proposition}
\begin{proof}
Indeed, for any $\eps>0$ we have $M(\eps , \Omega) \geq \mu_1(\Omega) / \eps$. Now substitute into the expression for and let $\eps$ tend to zero. We
get $\omega_{D}(\Omega) \ge \mu_{1}(\Omega)$. In order to show the equality for $\Omega=V_\rho$ being a sublevel set of $p$, we shall prove a slightly
more general statement: {\it Let $\Omega \subset B$ consist of $s$ closed intervals. Then for $s\leq M_D$ we have $\omega_{D}(\Omega) = \mu_{1}(\Omega)$}.
Indeed, let $\eps >0$ be given. We cover each of these subinterval $\Delta_i, \ i=1,\ldots,s$ of $\Omega$ by the adjacent $\eps$-intervals $Q_{\eps}$
starting with the left endpoint. Since all the adjacent $\eps$-intervals, except possibly one, are inside $\Delta_i$, their number doesn't exceed
$|\Delta_i| / \eps + 1$. Thus, we have $M(\eps,\Omega) \le s + \mu_1 (\Omega) / \eps,$ and therefore
$$
\eps[M(\eps,\Omega)-M_D]\le \eps[s + \mu_1 (\Omega) / \eps-M_D]\le \mu_1 (\Omega),
$$
if $s\leq M_D$. Since this inequality holds for each $\eps>0$, we conclude that $\omega_{D}(\Omega) \le \mu_1 (\Omega)$. \end{proof}

\smallskip

\noindent{\bf Remark} It looks plausible that the equality in Proposition \ref{Arithm} happens if {\it and only if} $\Omega=V_\rho$ is a sublevel set of $p$,
i.e. it consist of $s$ closed intervals, with $s\leq M_D$. Indeed, for one interval $\Delta_i$ if we take $\eps$ smaller than, but very close to
$|\Delta_i| / n,$ then we have $M(\eps,\Delta_i)$ very close to $|\Delta_i| / \eps + 1$. For two intervals, if their lengths are commeasurable, in exactly
the same way we can find $\eps$ in such a way that $M(\eps,\Delta_i \cup \Delta_j)$ very close to $(|\Delta_i|+ |\Delta_j|) / \eps + 2.$ If the lengths are
not commeasurable, we still can get the same result, using the density of the integer multiples of an irrational angle on the unit circle. Presumably, this
reasoning can be extended to any $s$, providing $\eps>0$ for which $M(\eps,\Omega)$ is very close to $\mu_1(\Omega) / \eps + s$. So if $s > M_D$, for this
specific $\eps$ we get $\eps[M(\eps,\Omega)-M_D]\geq \eps[\mu_1(\Omega) / \eps + s - M_D]> \mu_1(\Omega).$ Hence $\omega_{D}(\Omega) > \mu_{1}(\Omega)$.

\subsubsection{Subsets $\Omega$ dense ``in resolution $\eps$"}\label{dense.in.eps}

Here we show that the role of the frequency bound in the results above decreases as the discrete subset $\Omega \subset B$ becomes denser. For 
$\Omega \subset B$ and for $\eps >0$ we define the ``measure $\mu_1(\eps,\Omega)$
of $\Omega$ in resolution $\eps$" as the minimal possible measure of the coverings of $\Omega$ with $\eps$-intervals.

\begin{proposition}
For each diagram $D$ and for any $\eps >0$ the metric span $\omega_D(\Omega)$ satisfies
$$
\omega_D(\Omega) \ge \mu_1(\eps,\Omega) \left( 1-{{\eps M_D}\over {\mu_1(\eps,\Omega)}} \right)
$$
\end{proposition}

\begin{proof}
By the definition $\omega_D(\Omega)\ge \eps [M(\eps,\Omega)-M_D]$. Clearly, $M(\eps,\Omega)\ge {1\over \eps}\mu_1(\eps,\Omega)$. Hence $\omega_D(\Omega)\ge \mu_1(\eps,\Omega) - \eps M_D$.
\end{proof}

So if in a small resolution $\eps$, the measure $\mu := \mu_1(\eps,\Omega) > 0$ then we restore the original Tur\'an-Nazarov inequality for $\Omega$, with a correction factor $1- {{\eps M_D}\over {\mu}}$, with $M_D$ being
the frequency bound.

\subsubsection{Combining the discrete and positive measure cases}\label{comb.cases}

Let a diagram $D$ be fixed, and let $\Omega = \Omega_1 \cup \Omega_2 \subset B$, with $\Omega_1$ a set of a positive measure $\mu$, and $\Omega_2$ a discrete set. We assume that the sets $\Omega_1$ and $\Omega_2$ are
$2{{\mu_1(B)}\over M_D}$-separated, where $M_D$ is the frequency bound for $D$.

\begin{proposition}
$\omega_D(\Omega)\ge \mu + \omega_D(\Omega_2)$
\end{proposition}

\begin{proof}
By the definition $\omega_D(\Omega) = \sup_\eps \eps [M(\eps,\Omega) - M_D]$, and this supremum is achieved for $\eps \le {{\mu_1(B)}\over M_D}$. Indeed, otherwise $M(\eps,\Omega)-M_D$ would be negative. Hence by the
separation assumption we have $M(\eps,\Omega) = M(\eps,\Omega_1) + M(\eps,\Omega_2)$ and therefore $\omega_D(\Omega) = \sup_\eps \eps(M(\eps,\Omega_1)+ M(\eps,\Omega_2)-M_D)\ge \mu_1(\Omega_1)+\omega_D(\Omega_2)$.
\end{proof}

So in situations as above Theorem \ref{d.turan} improves the original Tur\'an-Nazarov inequality, and the frequency bound applies only to the discrete part of $\Omega$.

\subsubsection{Interpolation with exponential polynomials}\label{exp.interp}

This is a classical topic starting at least with \cite{P} and actively studied today in connection with numerous applications. Theorems \ref{d.turan}, \ref{d.turan1}, 
\ref{d.turan2} connect the Tur\'an-Nazarov inequality on $\Omega \subset B$ with estimates for the robustness of the interpolation from $\Omega$ to $B$. In particular,
 they provide robustness estimates in solving the ``generalized Prony system" for non-uniform samples. See \cite{S,Bat.Yom,Bat.Sar.Yom} for some initial results in this direction.

\section{Multi-dimensional case}\label{multi.dim}

In this section we consider the version of Tur\'an-Nazarov inequality for quasipolynomials in one or several variables due to A. Brudnyi \cite[Theorem 1.7]{B}. We provide a strengthening of this result in the same lines as
above: The Lebesgue measure is replaced with an appropriate ``metric span". First, let us recall some definitions.

\begin{definition}
Let $f_1, \ldots, f_k \in (\C^n)^*$ be a pairwise different set of complex linear functionals $f_j$ which we identify with the scalar products $f_j\cdot z, \ z=(z_1,\ldots,z_n)\in \C^n$. We shall write
$$
f_j=a_j+i b_j
$$
A quasipolynomial is a finite sum
$$
p(z) = \sum_{j=1}^k p_j(z) e^{f_j\cdot z}
$$
where $p_j \in \C[z_1,\ldots,z_n]$ are polynomials in $z$ of degrees $d_j$. The degree of $p$ is $m = \deg p = \sum_{j=1}^k (d_j + 1)$. Following A.Brudnyi \cite{B}, we introduce the exponential type of $p$
$$
t(p) = \max_{1\le j \le k} \max_{z\in B_c(0,1)} |f_j\cdot z|
$$
where $B_c(0,1)$ is the complex Euclidean ball of radius $1$ centered at $0$.
\end{definition}

Below we consider $p(x)$ for the real variables $x=(x_1,\ldots,x_n)\in \R^n$.

\begin{theorem}[\cite{B}]\label{brudnyi}
Let $p$ be a quasipolynomial with parameters $n,m,k$ defined on $\C^n$. Let $B \subset \R^n$ be a convex body, and let $\Omega \subset B$ be a measurable set. Then
$$
\sup_{B} |p| \le \left( \frac{c n \mu_n(B)}{\mu_n(\Omega)} \right)^{\ell} \cdot \sup_{\Omega} |p|
$$
where $\ell= (c(m,k) + (m-1)\log(c_1\max\{1,t(p)\}) + c_2 t(p) \diam(B))$, and $c, c_1,c_2$ are absolute positive constants, and $c(k,m)$ is a positive number depending only on $m$ and $k$.
\end{theorem}

Generalizing this result of Brudnyi, we follow the arguments described in Sections \ref{intro} and \ref{one.dim} above, and \cite{Y}.

\subsection{Covering number of sublevel sets}\label{covering.number}

For a relatively compact $A \subset \R^n$, the covering number $M(\eps,A)$ is defined now as the minimal number of $\eps$-cubes $Q_{\eps}$ covering $A$ (which are translations of the standard $\eps$-cubes $Q_{\eps}^n :=
[0,\eps]^n$).

\begin{lemma}
\begin{align*}
q(x) & := |p(x)|^2 \\
& =\sum_{0\le i \le j \le k} e^{\langle a_i+a_j, x\rangle } \big[ P_{i,j}(x) \sin \langle b_i-b_j , x \rangle + Q_{i,j}(x) \cos \langle b_i-b_j , x \rangle \big]
\end{align*}
is a real exponential trigonometric quasipolynomial with $P_{i,j},Q_{i,j}$ real polynomials in $x$ of degree $d_i + d_j$, and at most $\kappa: = k(k+1)/2$ exponents, sinus and cosinus elements.
\end{lemma}

\begin{proof}
By repeating word by word the proof of Lemma \ref{abs.val.p} above, the proof is completed.
\end{proof}

Clearly, all the partial derivatives ${{\partial q(x)}\over {\partial x_j}}$ have exactly the same form. The following bound due to Khovanskii gives an estimate of the number of solutions of a system of real exponential
trigonometric quasipolynomials. More precisely, we have

\begin{theorem}[Khovanskii bound \cite{K}, Section 1.4] \label{thm:khovanskii}
Let $P_1 = \cdots = P_n = 0$ be a system of $n$ equations with $n$ real unknowns $x = x_1,\ldots,x_n$, where $P_i$ is polynomial of degree $m_i$ in $n+k+2p$ real variables $x$, $y_1,\ldots,y_k$, $u_1,\ldots,u_p$, $v_1,
\ldots, v_p$, where $y_i = \exp\langle a_j,x\rangle$, $j=1,\ldots,k$ and $u_q = \sin\langle b_q,x\rangle$, $v_q = \cos\langle b_q,x\rangle$, $q=1,\ldots,p$. Then the number of non-degenerate solutions of this system in the
region bounded by the inequalities $|\langle b_q, x\rangle| < \pi/2$, $q=1,\ldots,p$, is finite and less than
$$
m_1 \cdots m_n \left( \sum m_i + p + 1\right)^{p+k} 2^{p+(p+k)(p+k-1)/2}
$$
\end{theorem}

Let us denote the vectors $b_i-b_j \in \R^n$ by $b_{i,j}$ and let $\lambda := \max \Vert b_{i,j} \Vert$ be the maximal frequency in $q$. The next lemma is a simple consequence of Khovanskii bound:

\begin{lemma}\label{khovanskii}
Let $V$ be a parallel translation of the coordinate subspace in $\R^n$ generated by $x_{j_1},\dots,x_{j_s}$. Then the number of non-degenerate real solutions in $V \cap Q_\rho^n$ of the system
$$
{{\partial q(x)}\over {\partial x_{j_1}}}= \cdots = {{\partial q(x)}\over {\partial x_{j_s}}}= 0
$$
is at most $\hat{C}_s \lambda^s$, where
$$
\hat{C}_s = ({2\over \pi}\sqrt s \rho)^s \prod_{r=1}^s (d_{j_r} + d_{i_r})\left( \sum_{r=1}^s d_{j_r} + d_{i_r} + 2\kappa + 1\right)^{2\kappa} 2^{\kappa+(2\kappa)(2\kappa-1)/2} .
$$
\end{lemma}


\begin{proof}
The following geometric construction is required by the Khovanskii bound: Let $Q_{i,j} = \{ x \in \R^n, |\langle b_{i,j}, x\rangle| \leq {\pi\over 2}\}$ and let $Q= \bigcap_{0\le i \le j \le k} Q_{i,j}$. For any $B \subset
\R^n$ we define $M(B)$ as the minimal number of translations of $Q$ covering $B$. For an affine subspace $V$ of $\R^n$ we define $M(B\cap V)$ as the minimal number of translations of $Q\cap V$ covering $B\cap V$. Notice that
for $B=Q^n_r$, a cube of size $r$, we have $M(Q^n_r)\leq ({2\over \pi}\sqrt n r \lambda)^n$. Indeed, $Q$ always contains a ball of radius ${\pi \over {2\lambda}}$. Now, applying the Khovanskii bound \ref{thm:khovanskii} on
the system
$$
{{\partial q(x)}\over {\partial x_{j_1}}}= \cdots = {{\partial q(x)}\over {\partial x_{j_s}}}= 0
$$
we get that the number of non-degenerate real solutions in $V \cap Q_\rho^n$ is at most
\begin{align*}
({2\over \pi}\sqrt s \rho \lambda)^s \prod_{r=1}^s (d_{j_r} + d_{i_r})\left( \sum_{r=1}^s d_{j_r} + d_{i_r} + 2\kappa + 1\right)^{2\kappa} 2^{\kappa+(2\kappa)(2\kappa-1)/2}
\end{align*}
\end{proof}

Let a quasipolynomial $p$ be as above. A sublevel set $A=A_{\rho}$ of $p$ is defined as $A=\{ x \in \R^n : |p(x)| \le \rho \}$. The following lemma extends to the case of sublevel sets of exponential polynomials the result
of Vitushkin \cite{V} for semi-algebraic sets. It can be proved using a general result of Vitushkin in \cite{V} through the use of ``multi-dimensional variations". However, in our specific case the proof below is much
shorter and it produces explicit (``in one step") constants.

\begin{lemma}\label{vitushkin}
For any $1\ge \eps>0$ we have
$$
M(\eps , A \cap Q^n_1) \le C_0 + C_1 \left({1 \over \eps}\right) + \cdots + C_{n-1} \left({1 \over \eps}\right)^{n-1} + \mu_n(A) \left({1 \over \eps}\right)^n
$$
where $C_0 , \ldots , C_{n-1}$ are positive constants, which depend only on $k, d_i$ and the maximal frequency $\lambda$ of the quasipolynomial $p$.
\end{lemma}

\begin{proof}
The sublevel set $A_\rho$ is defined via the real exponential trigonometric quasipolynomial $q(x)=|p(x)|^2$, i.e. $A = A_{\rho}(p) = \{ x \in Q_1^n : q(x) \le \rho^2 \}$. Let us subdivide $Q_1^n$ into adjacent $\eps$-cubes
$Q_{\eps}$ with respect to the standard Cartesian coordinate system. Each $Q_{\eps}$ having a nonempty intersection with $A$, is either entirely contained in $A$, or it intersects the boundary $\partial A$ of $A$. Certainly,
the number of those boxes $Q_{\eps}$, which are entirely contained in $A$, is bounded by $\mu_n(A) / \mu_n(Q_{\eps}) = \mu_n(A) / \eps^n$. In the other case, where $Q_{\eps}$ intersects $\partial A$, it means that there
exist faces of $Q_{\eps}$ that have a non-empty intersection with $\partial A$. Among all these faces, let us take the one with the smallest dimension $s$. In other words, there exists an $s$-face $F$ of the smallest
dimension $s$ that intersects $\partial A$, for some $s=0,1,\ldots,n$. Let us fix an $s$-dimensional affine subspace $V$, which corresponds $F$. Then $F$ contains completely some of the connected components of $A\cap V$,
otherwise $\partial A$ would intersect a face of $Q_{\eps}$ of a dimension strictly less than $s$. Clearly, inside each compact connected component of $A\cap V$ there is a critical point of $q$, which is defined by the
system of equations ${{\partial q(x)}\over {\partial x_{j_1}}}= \cdots = {{\partial q(x)}\over {\partial x_{j_s}}}= 0$ (assuming that V is a parallel translation of the coordinate subspace in $\R^n$ generated by
$x_{j_1},\dots,x_{j_s}$). After a small perturbation of $q$ we can always assume that all such critical points are non-degenerate. Hence by Lemma \ref{khovanskii} the number of these points, and therefore of the boxes
$Q_{\eps}$ of the considered type, is bounded by $\hat C_s \lambda^s$. According to the partitioning construction of $Q_1^n$, we have at most $\left( \frac {1}{\eps} + 1 \right)^{n-s}$ $s$-dimensional affine subspaces with
respect to the same $s$ coordinates. On the other hand, the number of different choices of $s$ coordinates is $\binom{n}{s}$. It means the number of boxes that have an $s$-face $F$, which contains completely some connected
component of $A\cap V$, is at most $\binom{n}{s}\cdot \left( \frac {1}{\eps} + 1 \right)^{n-s}\hat C_s \lambda^s$, which does not exceed, assuming $\eps \le 1$, the constant $C_{n-s}:= \binom{n}{s}2^{n-s}\hat C_s
\lambda^s(\frac {1}{\eps})^{n-s}$. Note that $C_0$ is the bound on the number of boxes that contain completely some of the connected components of $A$. Thus, we have
$$
M(\eps , A) \le C_0 + C_1 \left({1 \over \eps}\right) + \cdots + C_{n-1} \left({1 \over \eps}\right)^{n-1} + \mu_n(A) \left({1 \over \eps}\right)^n
$$
This completes our proof.
\end{proof}

\section{Metric span and generalized Brudnyi's inequality}

Let $p$ be a quasipolynomial as above, with the parameters $n, k , d_j$. These parameters, together with the maximal frequency $\lambda$ of $p$ form the multi-dimensional diagram $D$ of $p$. Notice that in contrast to the
one-dimensional case (and with Theorem \ref{brudnyi}) we restrict ourselves to the unit box $Q^n_1$. So $B$ does not appear in the diagram. For a given $0 < \eps \le 1$ let us denote by $M_D(\eps)$ the quantity
$M_D(\eps)=\sum_{j=0}^{n-1} C_j ({1\over \eps})^j$, where $C_0 , \ldots , C_{n-1}$ are the constants from Lemma \ref{vitushkin}. Extending the terminology from the one-dimensional case above, we call $M_D(\eps)$ the
``frequency bound" for $D$. Note that the constants $C_j$ depend only on the parameters $n, k, d_i$ and on the maximal frequency $\lambda$ of the quasipolynomial $p$. By Lemma \ref{vitushkin} for any sublevel set $A_\rho$ of
$p$ we have
$$
M(\eps , A) \le M_D(\eps) + \mu_n(A) \left( \frac{1}{\eps} \right)^n
$$
Now for any subset $\Omega \subset Q_1^n$ we introduce the metric span $\omega_D$ of $\Omega$ with respect to a given diagram $D$ as follows:

\begin{definition}
For a subset $\Omega \subset \R^n$ the metric span $\omega_D$ is defined as
$$
\omega_D(\Omega) = \sup_{\eps > 0} \eps^n [M(\eps , \Omega) - M_D(\eps)] .
$$
\end{definition}

\begin{lemma}\label{metric span}
Let $A \subset Q_1^n$ be a sublevel set of a real quasipolynomial with the diagram $D$. Then for any $\Omega \subset A$ we have
$$
\mu_n (A) \ge \omega_D(\Omega) .
$$
\end{lemma}

\begin{proof}
This fact follows directly from Lemma \ref{vitushkin}. Indeed, for any $\eps>0$ we have
$$
M(\eps , \Omega) \le M(\eps , A) \le M_D(\eps) + \mu_n(A) \left( \frac{1}{\eps} \right)^n .
$$
Consequently, for any $\eps >0$ we have $ \mu_n(A) \ge \eps^n [M(\eps , \Omega) - M_D(\eps)]$. Now, we can take the supremum with respect to $\eps$.
\end{proof}

For some examples and properties of sets in $\R^n$ with positive metric span, see \cite[Section 5]{Y}. Here we mention only that for a measurable $\Omega \subset \R^n$ we always have $\omega_D(\Omega)\ge \mu_n(\Omega)$. The
proof is exactly the same as in the remark after Theorem \ref{d.turan}. Now we can prove our generalization of Brudnyi's Theorem \ref{brudnyi} above.

\begin{theorem}\label{d.brudnyi}
Let $p$ be as above and let $\Omega \subset Q^n_1$. Then
$$
\sup_{Q^n_1} |p| \le \left( \frac{c n \mu_n(B)}{\omega_D(\Omega)} \right)^{\ell} \cdot \sup_{\Omega} |p|.
$$
\end{theorem}

\begin{proof}
Let $\hat \rho := \sup_{\Omega} |p|$. For the sublevel set $A_{\hat \rho}$ of the quasipolynomial $p$ we have $\Omega \subset A_{\hat \rho}$. By Lemma \ref{metric span} we have $\mu_n (A_{\hat \rho}) \ge \omega_D(\Omega)$.
Now since $p$ is bounded in absolute value by $\hat\rho$ on $A_{\hat \rho}$ by definition, we can apply Theorem \ref{brudnyi} with $B=Q_1^n$ and $A_{\hat \rho}$. This completes the proof.
\end{proof}

\address

\end{document}